\documentclass[11pt]{article}
\usepackage{graphicx}
\usepackage{tikz}
\usepackage{amsfonts}
\usepackage{amsbsy}
\usepackage{amssymb}
\usepackage{amsmath,amsthm}
\usepackage{verbatim}
\usepackage{amsfonts}
\usepackage{pst-all}
\usepackage{pstricks}
\usepackage[T1]{fontenc}
\usepackage{multicol}
\usepackage{color}
\usepackage[autostyle]{csquotes}
\usepackage{pgf}
\usepackage{mathtools}
\DeclarePairedDelimiter\ceil{\lceil}{\rceil}
\DeclarePairedDelimiter\floor{\lfloor}{\rfloor}

\newtheorem{thm}{Theorem}
\newtheorem{prop}{Proposition}[section]
\newtheorem{cor}[thm]{Corollary}
\newtheorem{defn}[thm]{Definition}
\newtheorem{obs}[prop]{Observation}

\newtheorem{conj}[thm]{Conjecture}

\newtheorem{rem}[prop]{Remark}
\newtheorem{exm}[thm]{Example}
\newtheorem{prob}[thm]{Problem}

\newcommand{\Z}{\mathbb{Z}}

\RequirePackage{pgfkeys}
\pgfkeys{/ytableau/options/.is family}

\pgfkeys{/ytableau/options, boxsize/.value required,  boxsize/.code = {%
\pgfkeysalso{nosmalltableaux}%
 \compare@YT{#1}{normal}%
 \ifeq@YT
 \xdef\macro@boxdim@YT{\expandonce@YT\boxdim@normal@YT}%
 \else
 \xdef\macro@boxdim@YT{#1}%
 \fi
 } }
 \pgfkeys{/ytableau/options, textmode/.value forbidden,
 textmode/.code = \set@textmode@YT,
mathmode/.value forbidden,
mathmode/.code = \set@mathmode@YT,
 }
 
\def\set@mathmode@YT{
 \gdef\skipin@YT{$}
 \gdef\skipout@YT{$}
 \def\smallfont@YT{\scriptstyle} } 
 \set@mathmode@YT

\title{Group-annihilator graphs realised by finite abelian groups and its properties}
\author{Eshita Mazumdar\footnote{Stat-Math Unit, ISI Bengaluru, Karnataka, India. Email: eshita\_vs (at) isibang.ac.in} and 
Rameez Raja\footnote{Department of Mathematics, NIT Srinagar, Jammu and Kashmir, India. Email: rameeznaqash (at) nitsri.ac.in}}

\begin{document}
\maketitle
\begin{abstract} 
Let $G$ be a finite abelian group viewed a $\mathbb{Z}$-module and let $\mathcal{G} = (V, E)$ be a simple graph. In this paper, we consider a graph $\Gamma(G)$ called as a \textit{group-annihilator} graph. The vertices of $\Gamma(G)$ are all elements of $G$ and two distinct vertices $x$ and $y$ are adjacent in $\Gamma(G)$ if and only if $[x : G][y : G]G = \{0\}$, where $x, y\in G$ and $[x : G] = \{r\in\mathbb{Z} : rG \subseteq \mathbb{Z}x\}$ is an ideal of a ring $\mathbb{Z}$. We discuss in detail the graph structure realised by the group $G$. Moreover, we study the creation sequence, hyperenergeticity and hypoenergeticity of group-annihilator graphs. Finally, we conclude the paper with a discussion on Laplacian eigen values of the group-annhilator graph. We show that the Laplacian eigen values are representatives of orbits of the group action: $Aut(\Gamma(G)) \times G \rightarrow G$.   
\end{abstract}


2010 AMS Classification Code: 05E16, 05C25, 05C50.  
\section{Introduction}
This research article is an exploration of the relationship between the group theoretical properties of an abelian group $G$ and combinatorial (graph theoretical) properties of a graph realised by $G$. There is an intimate relationship between groups and graphs. For example, any graph $\Gamma$ gives rise to its automorphism group. On the other hand, any group with its generating set give rise to a Cayley graph. There are number of constructions of graphs from groups. Some of the graphs arising from groups are: power graph \cite{CS}, generating graph \cite{LS}, intersection graph \cite{ZB} and the commuting graph \cite{BF}. These graphs were introduced to study the information that is contained in the graph about the group. It is to be noted that the generating graph of a simple group is being studied to get an insight that might eventually lead us to a new proof of the classification of simple groups. 

The investigation of combinatorial and graph theoretical properties have also been studied in rings and modules. Beck \cite{Bk} introduced the concept of associating a graph to a commutative ring $R$. He associated a simple graph to $R$, which is known as a zero-divisor graph in the literature. The vertices of a zero-divisor graph are all elements of $R$ with two distinct vertices $x$ and $y$ being adjacent in the graph if and only if $xy = 0$. For more on the zero-divisor graph structure of a ring, see \cite{AdLs, SRR, SR, SRR1, Rd}. This concept was extended to modules over commutative rings in \cite{SR, RR}.

This paper is organised as follows. In section 2, we introduce the \textit{group-annihiltor} graph realised by group $G$ and provide some examples. Section 3 is devoted to study in detail the \textit{group-annihilator}  graph structure realised by groups of different ranks. In section 4, we show that a group-annihilator graph realised by the group $G = \mathbb{Z}/p^{\alpha}\mathbb{Z}$ is a threshold graph. Furthermore, we discuss creation sequence, hyperenergeticity and hypoenergeticity of group-annihilator graphs and assert that it is pointless to determine the classes of hyperenergetic graphs. We show that for each prime $p \geq 7$ there exists a connected threshold graph with $p^2$ vertices such that its energy lies in the interval $(a, b)$, where $a$ is the energy of a group-annihilator graph realised by the group $G$ and $b$ is the energy of a complete graph on $p^2$ vertices. Finally, in Section 5, we discuss Laplacian eigen values of the group-annihilator graph $\Gamma(G)$. We prove a very interesting property of the graph $\Gamma(G)$. In fact, we show that the Laplacian eigen values of $\Gamma(G)$ are representatives of orbits of the group action: $Aut(\Gamma(G)) \times G \rightarrow G$.

\section{Preliminaries}
It is clear that $\mathbb{Z}$ is a commutative ring with unity under usual addition and multiplication. 
Consider a finite abelian non-trivial group $G$ with identity element $0$ and view $G$ as a $\mathbb{Z}$-module. For $a\in G$, set $$[a : G] =\{x\in \mathbb{Z} ~|~ xG\subseteq \mathbb{Z}a\},$$ 
which clearly is an ideal of $\mathbb{Z}.$ For $a \in G$, $G/\mathbb{Z}a$ is a $\mathbb{Z}$-module. So  $[a : G]$ is a annihilator of $G/\mathbb{Z}a$. We call $[a:G]$ as $a$-annihilator of $G.$ Also, we call an element $a$ as an \textit{ideal-annihilator} of $G$ if there exists a non-zero element $b$ of $G$ such that $[a : G][b : G]G = \{0\}$, where $[a : G][b : G]$ denotes the product of ideals of $\mathbb{Z}$. The element $0$ is a trivial ideal-annihilator of $G$, since $[0 : G][b : G]G = ann(G)[b : G]G = \{0\}$, $ann(G)$ is an annihilator of $G$ in $\mathbb{Z}$. 

Given an abelian group $G$, we define the \textit{group-annihilator} graph to be the graph $\Gamma(G) = (V(\Gamma(G))$, $E(\Gamma(G)))$ with vertex set $V(\Gamma(G))= G$ and for two distinct $a, b\in V(\Gamma(G))$, the vertices $a$ and $b$ are adjacent in $\Gamma(G)$ if and only if $[a : G][b : G]G = \{0\}$, that is, $E(G) = \{(a, b)\in G \times G : [a : G][b : G]G = \{0\}\}$. 

By the definition of group-annihilator graph, it appears that the vertex $0$ is adjacent to all vertices of the graph and the structure of $\Gamma(G)$ is simple, that is, $\Gamma(G)$ is without self loops and parallel edges. Furthermore, the eccentricity of 0 (distance between $0$ and farthest verrtex from $0$) in  $\Gamma(G)$ is 1, which is in fact the minimum eccentricity of the graph. Thus $0$ is the central vertex of the graph. The maximum eccentricity of $\Gamma(G)$ is atmost 2, since the distance between any two vertices $\Gamma(G)$ is atmost 2. If some graph $\mathcal{G}$ contains a cycle, then the girth (length of the smallest cycle) and diameter of $\mathcal{G}$ denoted by $gr(\mathcal{G})$ and $diam(\mathcal{G})$, respectively, are related by the inequality $gr(\mathcal{G}) \leq 2diam(\mathcal{G}) + 1$. It follows by the inequality that $gr(\Gamma(G)) \leq 5$. However, it is clear from the structure of the group-annihilator graph that if $\Gamma(G)$ contains a cycle, then  $gr(\Gamma(G)) = 3$.\\

Let $\mathcal{G}_1$ and $\mathcal{G}_2$ be two simple connected graphs, recall a mapping $\phi: \mathcal{G}_1 \rightarrow \mathcal{G}_2$ is a \textit{homomorphism} if it perserves edges, that is, for any edge $(u, v)$ of $\mathcal{G}_1$, $(\phi(u), \phi(v))$ is an edge of $\mathcal{G}_2$, where $u, v\in V(\mathcal{G}_1)$. A homomorphism  $\phi: \mathcal{G}_1 \rightarrow \mathcal{G}_2$ is faithful when there is an edge between two preimages $\phi^{-1}(u)$ and $\phi^{-1}(u)$ such that $(u, v)$ is an edge of $\mathcal{G}_2$, a faithful bijective homomorphism is an \textit{isomorphism} and in this case we write $\mathcal{G}_1 \cong \mathcal{G}_2$. An isomorphism from $\mathcal{G}$ to itself is an \textit{automorphism} of $\mathcal{G}$, it is well known that set of automorphisms of $\mathcal{G}$ forms a group under composition, we denote the group of automorphisms of $\mathcal{G}$ by $Aut(\mathcal{G})$. Understanding the automorphism group of a graph is a guiding principle for understanding objects by their symmetries.

Let $\mathcal{G} = (V, E)$ be a simple graph. Consider the group action: $Aut(\mathcal{G}) ~acting~on ~V(\mathcal{G})$ by some permutation of $Aut(\mathcal{G})$. That is; 

\hskip .9cm \hskip .9cm \hskip .9cm \hskip .9cm \hskip .9cm \hskip .9cm $Aut(\mathcal{G}) \times V(\mathcal{G}) \rightarrow V(\mathcal{G})$, 

\hskip .9cm \hskip .9cm \hskip .9cm \hskip .9cm \hskip .9cm \hskip .9cm \hskip .9cm \hskip .3cm $\sigma(v) = u$,

\noindent where $\sigma \in Aut(\mathcal{G})$ and $v, u\in V(\mathcal{G})$ are any two vertices of $\mathcal{G}$. We call this group action as \textit{symmetric action}. The authors in \cite{KA} have considered the action: $Aut(G) \times G \rightarrow G$, where $Aut(G)$ is an automorphism group of $G$ and studied the $Aut(G)$-orbits in $G$. They exhibited $Aut(G)$-orbits in $G$ as elements of a fundamental partially ordered set and investigated an interesting interplay of properties of partially ordered sets and finite abelian groups. 

There is an advantage for knowing the orbits of group action: $Aut(\Gamma(G)) \times V(\Gamma(G)) \rightarrow V(\Gamma(G))$, because we get some structural information about some elements of group $G$ from $\Gamma(G)$, and as a consequence we do not consider all elements of $G$ to decode the symmetry of $\Gamma(G)$. We also explore this information to reveal some interesting spectral propeties of the graph $\Gamma(G)$.

Let $Aut(G)$ be an automorphism group of $G$. Under the action of $Aut(G)$ on $G$, we have the set $Aut(G)\setminus G$ of $Aut(G)$-orbits in $G$.  For each prime $p$, let $G_p$ denote the elements of $G$ which are annihilated by some power of $p$. Then $G$ is the direct sum of $p$-subgroups $G_p$. On the other hand, every finite abelian $p$-group is isomorphic to the group of type $\lambda = (\lambda_1, \lambda_2, \cdots, \lambda_r)$:
\begin{center}
$G_{\lambda, p} = \mathbb{Z}/p^{\lambda_1}\mathbb{Z} \oplus \mathbb{Z}/p^{\lambda_2}\mathbb{Z} \oplus \cdots \oplus \mathbb{Z}/p^{\lambda_r}\mathbb{Z}$,
\end{center}
where $\lambda = \lambda_1 \geq \lambda_2 \geq \cdots \geq \lambda_r$ represents a unique partition. Therefore the automorphism group action is on each element of a finite abelian $p$-groups.

There are some well known formulae for the cardinality of the set $Aut(G_{\lambda, p})$ $\setminus$ $G_{\lambda, p}$. Miller \cite{M} proved that;

\begin{center}$|Aut(G_{\lambda, p})$ $\setminus$ $G_{\lambda, p}| = (\lambda_r + 1)\prod\limits_{i = 1}^{r - 1}(\lambda_i - \lambda_{i + 1} + 1)$,\end{center}
where $Aut(G) \setminus G = \prod\limits_{p} Aut(G_p) \setminus G_p$.\\

Schwachh\"ofer and Stroppel \cite{SS} also proved that if $\tau_1 < \tau_2 < \cdots < \tau_s$ are $s$ distinct natural numbers occuring in the partition $\lambda$, then;
\begin{center}
$|Aut(G_{\lambda, p})$ $\setminus$ $G_{\lambda, p}| = \sum\limits_{k = 0}^{s} \sum\limits_{1\leq \tau_1 < \tau_2 \cdots \tau_k\leq s} \tau_{i_{k}}\prod\limits_{j = 1}^{k - 1}(\tau_{i_j} - \tau_{i_{j+1}} - 1)$.
\end{center}

\section{On group-annihilator graphs realised by G}
This section is devoted to study the group-annihilator graphs realised by groups of different ranks such as: 
$$(\mathbb{Z}/p\mathbb{Z})\times (\mathbb{Z}/p\mathbb{Z}) \cdots \times (\mathbb{Z}/p\mathbb{Z}), (\mathbb{Z}/p^{\alpha}\mathbb{Z}) \text{ and } \Z/p^{\alpha}\Z \times\Z/p^{\beta}\Z \times \Z/p^{\gamma}.$$ 
Let $\lambda = (\lambda_1, \lambda_2, \cdots, \lambda_r)$ be a partition of $n$ denoted by $\lambda \vdash n$. For any  $\mu \vdash n$, we have an abelian group of order $p^n$ and conversely every abelian group corresponds to some partion of $n$. In fact, if $H_{\mu, p} =  \mathbb{Z}/p^{{\mu}_1}\mathbb{Z} ~\oplus~ \mathbb{Z}/p^{{\mu}_2}\mathbb{Z} ~\oplus~ \cdots ~\oplus~ \mathbb{Z}/p^{{\mu}_r}\mathbb{Z}$ is a subgroup of $G_{\lambda, p}$, then $\mu_1 \leq \lambda_1, \mu_2 \leq \lambda_2, \cdots, \mu_r \leq  \lambda_r$. If these inequalities holds we write $\mu \subset \lambda$, that is a \textquotedblleft containtment order\textquotedblright on partitions. For example, a $p$-group $\mathbb{Z}/p^{5}\mathbb{Z} ~\oplus~ \mathbb{Z}/p\mathbb{Z} ~\oplus~ \mathbb{Z}/p\mathbb{Z}$ is of type $\lambda = (5, 1, 1)$. The possible types for its subgroup are: $(5, 1, 1), (4, 1, 1), (3, 1, 1), (2, 1, 1), (1, 1, 1), 2(5, 1), 2(4, 1), 2(3,1), 2(2, 1), 2(1, 1), (5), (4), (3), (2), 2(1)$.

Note that the types $(5, 1), (4, 1), (3,1), (2, 1), (1, 1)$ are appearing twice in the sequence of partitions for a subgroup.

For any $a \in (\mathbb{Z}/p\mathbb{Z})^n$, let $\lambda = (1, 1, \cdots, 1) = (1^n)$. A group of type $\lambda$ is nothing but the $\mathbb{Z}/p\mathbb{Z}$-vector space $\mathbb{Z}/p\mathbb{Z} ~\oplus~ \mathbb{Z}/p\mathbb{Z} ~\oplus~ \cdots ~\oplus~ \mathbb{Z}/p\mathbb{Z}$. Its subgroups are of type $(1^r )$, where $0\leq r \leq n$. Clearly, the ideal associated with any $a \in (\mathbb{Z}/p\mathbb{Z})$  is $a\mathbb{Z}$. Therefore the group-annihilator graph realised by $(\mathbb{Z}/p\mathbb{Z})$ is a star graph with $p$ vertices.

Let $p$ be a prime and let $G =\Z/p^{\alpha}\Z$ be a cyclic group of order $p^{\alpha} $. Then, for $i \in [0, \alpha]$, $ \mathcal{O}_{\alpha,p^i} = \{p^ib (\bmod p^{\alpha})\mid b \in \Z, (b,p) =1\}$ is the $p^i$-th orbit of $G$. Moreover, for $0\leq i< j \leq \alpha,$ 
$$p^ib \equiv p^j b' (\bmod p^{\alpha})\text{ where } (b,p) =1 \text{ and } (b',p)=1.$$
Therefore, $p^{(j-i)}b' \equiv b (\bmod p^{(\alpha-i)}),$ which is a contradiction. Thus, for $i \neq j$, $\mathcal{O}_{\alpha,p^i} \cap  \mathcal{O}_{\alpha,p^j} = \emptyset$.

Any element $a \in \Z/p^{\alpha}\Z$ can be expressed as; $$a \equiv p^{\alpha-1}b_1 + p^{\alpha-2} b_2 +\cdots + p b_{\alpha-1 } +b_{\alpha} (\bmod p^{\alpha}),$$
where $b_i \in [0,p-1].$ If $a \in \mathcal{O}_{\alpha,1}$,
then $b_{\alpha} \neq 0.$ So, $|\mathcal{O}_{\alpha,1}| = p^{\alpha-1} (p-1) = \phi (p^{\alpha})$.

If $a' \in \mathcal{O}_{\alpha,p}$, then for some $a \in \mathcal{O}_{\alpha,1}$ $a' = pa$, that is, $b_{\alpha}\neq 0$, so $|\mathcal{O}_{\alpha,p}|
= \frac{\phi (p^{\alpha})}{p}$. Similarly, for $i \in [0,\alpha]$, we have
  $|\mathcal{O}_{\alpha,p^i}| = \frac{\phi(p^{\alpha})}{p^i}.$ 

\begin{prop} Let $G =\Z/p^{\alpha}\Z$ be a cyclic group of $p^{\alpha} $ order where $p$ is prime and $\alpha \ge 2$. Then for each $a\in G$ the $a-$annihilator of $G$ is $$[a: G] = p^i \Z \text{ for each } a \in  \mathcal{O}_{\alpha,p^i} \text{ and } i\in \{0,1,2\cdots,\alpha\}$$
\end{prop}
where $\mathcal{O}_{\alpha,p^i} = \{p^ib (\bmod p^{\alpha})\mid b \in \Z, (b,p) =1\}$ is the $p^i$-th orbit of the group $\Z/p^{\alpha}\Z.$ 

\begin{proof}
As $\mathcal{O}_{\alpha,p^i}$ is an orbit of $\Z/p^{\alpha} \Z$ so $$\Z/p^{\alpha}\Z = \dot\bigcup_{i=0}^{i=\alpha}\mathcal{O}_{\alpha,p^i}.$$
Clearly, $[0: G] = p^{\alpha}\Z ,$ where $\{0\} =  \mathcal{O}_{\alpha,p^{\alpha}}.$

For $a \in \mathcal{O}_{\alpha,p^i}$ where $i \in [0,\alpha-1]$, we have $a= p^i b$ for some $b\in \Z$ and $(b,p) =1$. Let $y =p^i k $ for some $k.$ Now for any $g \in G$, 
$yg = p^ib (b^{-1} kg) \in \Z a$. Theorefore, $p^i \Z \subset [a: G]. $ On the other hand, $y \in  [a:G]$ implies $yG \subset \Z a$, therefore $y= a k'$ for some integer $k'.$ So, $y = p^i k' b \in p^i \Z.$ 
\end{proof}

\begin{obs}\label{1}
If we consider the group $G = \Z/p^{\alpha}\Z$ for a prime $p$ and an integer $\alpha\ge 2$, then the group-annihilator graph is defined as follows:
$$\Gamma (G) = (V(\Gamma(G)), E(\Gamma(G))), \text{ where } V(\Gamma(G)) = \Z/p^{\alpha}\Z$$
and for $a,b \in G$, 
$$a \text{ is adjacent to } b \text{ iff } i+j \geq \alpha,$$ where
$a \in \mathcal{O}_{\alpha,p^i}$, $b \in 
\mathcal{O}_{\alpha,p^j}$. 
\end{obs}
Here, we present a brief description of the above observation. Clearly, by definition of the group-annihilator graph, an element $0$ of $G$ is adjacent to all vertices in $\Gamma (G)$, the relatively prime elements of $G$ are adjacent to $0$ only in $\Gamma (G)$. Furthermore, elements of the orbit $\mathcal{O}_{\alpha,p}$ are adjacent to $0$ and elements of the orbit $\mathcal{O}_{\alpha,p^{\alpha -1}}$, the elements of orbit $\mathcal{O}_{\alpha,p^2}$ are adjacent to $0$ and elements of orbits $\mathcal{O}_{\alpha,p^{\alpha -1}}$,  $\mathcal{O}_{\alpha,p^{\alpha -2}}$. Thus for $k \geq 1$, elements of the orbit
 $\mathcal{O}_{\alpha,p^k}$ are adjacent to elements of the orbits $\mathcal{O}_{\alpha,p^{\alpha -k}}$,  $\mathcal{O}_{\alpha,p^{\alpha -k + 1}}, \cdots, \mathcal{O}_{\alpha,p^{\alpha -1}}$.

 \begin{thm} Let $\alpha $ be a positive integer.
For the $p$-group $G = (\Z/p^{\alpha}\Z)^{\ell}  $ of rank $\ell\geq2,$ and $(a_1,\ldots,a_l)\in G$, the 
$(a_1,\ldots,a_l)$-annihilator of $G$ is $p^{\alpha}\Z.$ In particular the corresponding group-annihilator graph is a complete graph.

\end{thm}

\begin{proof}
Let $\mathcal{O}_{\alpha, p^i}$ be the $p^i$-th orbit of the group $\Z/p^{\alpha}\Z$ for $0\leq i\leq \alpha.$ 
Let $(a_1,\ldots,a_l)\in G$ and $a_i = p^{j_i}{a_i}' \in 
\mathcal{O}_{\alpha, p^{j_i}}$ where
$1\leq i\leq l$ and $j_k$ be the least among all $j_i.$
Clearly $p^{\alpha}(x_1,\ldots,x_l) \in Z(a_1,\ldots,a_l),$
for any $(x_1,\ldots,x_l)\in G.$
Now let $y\in [(a_1,\ldots,a_l):G]$. Choose an element $b_{k,r}$ of $G$ such that there exists a position $j_r$ with $r\neq k$ 
where it's value is $1,$ in other positions it takes the value $0.$ 
Therefore $yb_{k,r} = n (p^{i_1}a_1',....p^{i_l}a_l')$, for some integer $n$. This gives $p^{\alpha-j_k} \mid n$ and $y \equiv p^{\alpha-j_k+j_r}a_r' (\bmod p^{\alpha}).$ This completes the first part.

On the other hand, for any two elements $a,b\in G$, $[a:G][b:G]G =\{0\}$. So the associated group annihilator graph forms a complete graph.
\end{proof}

Here the action of $Aut(\Gamma((\mathbb{Z}/p\mathbb{Z})^{\ell}))$ on $(\mathbb{Z}/p\mathbb{Z})^{\ell}$ is transitive, since an automorphism of $\Gamma((\mathbb{Z}/p\mathbb{Z})^{\ell})$ map any vertex to any other vertex and this does not place any restriction on where any of the other $p^{\ell} - 1$ vertices are mapped, as they are all mutually connected in $\Gamma((\mathbb{Z}/p\mathbb{Z})^{\ell})$. This implies $Aut(\Gamma((\mathbb{Z}/p\mathbb{Z})^{\ell})) \setminus(\mathbb{Z}/p\mathbb{Z})^{\ell}$ is a single orbit of order $p^{\ell}$.

\begin{exm} Let $G = \mathbb{Z}/8\mathbb{Z}$. Consider the group action: $Aut(\Gamma(G)) ~ acting ~on ~G$. The orbits of this action are: $Orb(0) = \{0\}$, $Orb(1) = \{1, 3, 5, 7\}$, $Orb(2) = \{2, 6\}$, $Orb(4) = \{4\}$. The orbits of elements $3, 5, 7$ are same as the orbit of $1$ and the orbit of $6$ is same as the orbit of $2$.  Therefore, the group $G$ has $3$ orbits of nonzero elements under the action of $Aut(\Gamma(G))$ represented by $1, 2, 2^2$. Now using the concept mention above we have the following group-annihilator graph realised by $G$.\end{exm} 

\begin{align*}
\begin{pgfpicture}{7cm}{-3cm}{4cm}{2cm}
\pgfnodecircle{Node1}[fill]{\pgfxy(7,2)}{0.1cm}
\pgfnodecircle{Node2}[fill]{\pgfxy(4,-.5)}{0.1cm}
\pgfnodecircle{Node3}[fill]{\pgfxy(5,-.5)}{0.1cm}
\pgfnodecircle{Node4}[fill]{\pgfxy(6,-.5)}{0.1cm}
\pgfnodecircle{Node5}[fill]{\pgfxy(7, -.5)}{0.1cm}
\pgfnodecircle{Node6}[fill]{\pgfxy(8, -.5)}{0.1cm}
\pgfnodecircle{Node7}[fill]{\pgfxy(9,-.5)}{0.1cm}
\pgfnodecircle{Node8}[fill]{\pgfxy(10, -.5)}{0.1cm}
\pgfnodeconnline{Node1}{Node2}
\pgfnodeconnline{Node1}{Node3}
\pgfnodeconnline{Node1}{Node4}
\pgfnodeconnline{Node1}{Node5}
\pgfnodeconnline{Node1}{Node6}
\pgfnodeconnline{Node1}{Node7}
\pgfnodeconnline{Node1}{Node8}
\pgfnodeconnline{Node4}{Node5}
\pgfnodeconnline{Node5}{Node6}
\pgfputat{\pgfxy(6.9, 2.3)}{\pgfbox[left,center]{$\bar{0}$}}
\pgfputat{\pgfxy(3.8, -.9)}{\pgfbox[left,center]{$\bar{1}$}}
\pgfputat{\pgfxy(4.9, -.9)}{\pgfbox[left,center]{$\bar{3}$}}
\pgfputat{\pgfxy(5.9,-.9)}{\pgfbox[left,center]{$\bar{2}$}}
\pgfputat{\pgfxy(6.9, -.9)}{\pgfbox[left,center]{$\bar{4}$}}
\pgfputat{\pgfxy(7.9,-.9)}{\pgfbox[left,center]{$\bar{6}$}}
\pgfputat{\pgfxy(8.9, -.9)}{\pgfbox[left,center]{$\bar{5}$}}
\pgfputat{\pgfxy(10, -.9)}{\pgfbox[left,center]{$\bar{7}$}}
\pgfputat{\pgfxy(4.9, -1.8)}{\pgfbox[left,center]{$Figure \hskip .1cm 1. \hskip .2cm \Gamma(\mathbb{Z}/8\mathbb{Z})$}}
\end{pgfpicture}
\end{align*}

A finite $p$-group $G = \Z/p^{\alpha}\Z \times\Z/p^{\beta}\Z \times \Z/p^{\gamma}\Z $ of rank $3,$ can be expressed as; $$G = \dot\bigcup_{i=0}^{\alpha}\mathcal{O}_{\alpha,p^i} \times \dot\bigcup_{j=0}^{\beta}\mathcal{O}_{\beta,p^j}\times\dot\bigcup_{k=0}^{\gamma} \mathcal{O}_{\gamma,p^k}$$ where $\alpha \leq \beta  \leq \gamma$ and $\mathcal{O}_{s,p^i}$ is the $p^i$-th orbit of the group $\Z/p^s\Z.$\\

\begin{thm}
Let $G = \Z/p^{\alpha}\Z \times\Z/p^{\beta}\Z \times \Z/p^{\gamma}\Z $  be a $p$-group where
$\alpha < \beta <\gamma. $ Then for each $(a,b,c) \in G$ the $(a,b,c)$-annihilator of $G$ are following:
\begin{itemize}
\item Let $a \in \mathcal{O}_{\alpha, p^{\alpha}}$, 
\begin{itemize} 
\item For $b\in \mathcal{O}_{\beta, p^{\beta}},$
\begin{eqnarray*}
[(a,b,c) : G]&=& p^{\gamma} \mathbb{Z} \text{ for } c\in \mathcal{O}_{\gamma, p^{\gamma}},\\
&=&p^{\beta}\mathbb{Z} \text{ for } c \in \mathcal{O}_{\gamma, p^i} \text{ where } 0\leq i \leq \beta-1,\\
&=&p^{i}\mathbb{Z} \text{ for } c \in \mathcal{O}_{\gamma, p^i} \text{ where } \beta \leq i \leq \gamma-1.\end{eqnarray*}

\item For $b\in \mathcal{O}_{\beta, p^j}$ where $0\leq j \leq \beta -1,$
\begin{eqnarray*}
[(a,b,c) : G]&=& p^{\gamma} \mathbb{Z} \text{ for } c\in \mathcal{O}_{\gamma, p^{\gamma}},\\
&=& p^{\beta} \mathbb{Z} \text{ for } c\in \mathcal{O}_{\gamma, p^i}\text{ where } 0\leq i \leq j,\\
&=&p^{i+ \beta - j }\mathbb{Z} \text{ for } c \in \mathcal{O}_{\gamma, p^i} \text{ where } j +1 \leq i \leq \gamma -\beta + j,\\
&=&p^{\gamma}\mathbb{Z} \text{ for } c \in \mathcal{O}_{\gamma, p^i} \text{ where } \gamma - \beta +j +1 \leq i \leq \gamma- 1.\end{eqnarray*}
\end{itemize}

\item Let $a \in \mathcal{O}_{\alpha, p^k}$, where $0\leq k\leq \alpha-1$ 
\begin{itemize} 
\item For $b\in \mathcal{O}_{\beta, p^{\beta}},$
\begin{eqnarray*}
[(a,b,c) : G]&=& p^{\gamma} \mathbb{Z} \text{ for } c\in \mathcal{O}_{\gamma, p^{\gamma}},\\
&=&p^{\beta}\mathbb{Z} \text{ for } c \in \mathcal{O}_{\gamma, p^i} \text{ where } 0\leq i \leq \beta - \alpha+k,\\
&=&p^{i+\alpha-k}\mathbb{Z} \text{ for } c \in \mathcal{O}_{\gamma, p^i} \text{ where } \beta - \alpha + k +1 \leq i \leq \gamma - \alpha +k-1,\\
&=&p^{\gamma}\mathbb{Z} \text{ for } c \in \mathcal{O}_{\gamma, p^i} \text{ where } \gamma - \alpha +k \leq i \leq \gamma -1.\\\end{eqnarray*}

\end{itemize}

\begin{itemize} 
\item For $b\in \mathcal{O}_{\beta, p^{j}},$ where $0 \leq j \leq \beta -1.$
\begin{eqnarray*}
[(a,b,c) : G]&=& p^{\gamma} \mathbb{Z} \text{ for } c\in \mathcal{O}_{\gamma, p^{\gamma}},\\
&=&p^{\beta}\mathbb{Z} \text{ for } c \in \mathcal{O}_{\gamma, p^i} \text{ where } i \leq j, 0\leq i \leq \beta-\alpha \text{ and } 0 \leq i \leq \gamma - \beta +j,\\
&=&p^{\beta +i-j}\mathbb{Z} \text{ for } c \in \mathcal{O}_{\gamma, p^i} \text{ where } i > j, 0\leq i \leq \beta-\alpha \text{ and } 0 \leq i \leq \gamma - \beta +j,\\
&=&p^{i+\beta-j}\mathbb{Z} \text{ for } c \in \mathcal{O}_{\gamma, p^i} \text{ where } \gamma-\beta+j \ge i \geq j, \beta - \alpha < i
\text{ and } \beta - \alpha > j,\\
&=&p^{\alpha-k+i}\mathbb{Z} \text{ for } c \in \mathcal{O}_{\gamma, p^i} \text{ where } \gamma-\beta +j \ge i \geq j, \gamma -\alpha +k >i \text{ and }   j> \beta -\alpha+k,\\
&=&p^{\gamma}\mathbb{Z} \text{ for } c \in \mathcal{O}_{\gamma, p^i} \text{ where } \gamma-\beta+j \ge i \geq j, \gamma -\alpha +k \leq i \text{ and }   j> \beta -\alpha+k,\\
&=&p^{\beta + i-j}\mathbb{Z} \text{ for } c \in \mathcal{O}_{\gamma, p^i} \text{ where } \gamma-\beta+j\ge i \geq j \text{ and } \beta - \alpha < j\leq \beta -\alpha+k,\\
&=&p^{\beta }\mathbb{Z} \text{ for } c \in \mathcal{O}_{\gamma, p^i} \text{ where } i < j, \gamma-\beta+j \ge i >\beta -\alpha
\text{ and } j \leq \beta - \alpha +k,\\
&=&p^{\beta }\mathbb{Z} \text{ for } c \in \mathcal{O}_{\gamma, p^i} \text{ where } i < j, \gamma-\beta+j \ge i, \beta-\alpha+k >i >\beta -\alpha\text{ , } j > \beta - \alpha +k,\\
&=&p^{\alpha+i-k }\mathbb{Z} \text{ for } c \in \mathcal{O}_{\gamma, p^i} \text{ where } i < j , \gamma- \beta +j \ge i \text{ and } \gamma-\alpha+k > i \geq \beta -\alpha+k,\\
&=&p^{\gamma }\mathbb{Z} \text{ for } c \in \mathcal{O}_{\gamma, p^i} \text{ where } i < j, \gamma-\beta+j \ge i \geq \gamma -\alpha+k,\text{ and } j > \beta - \alpha +k,\\
&=&p^{\gamma}\mathbb{Z} \text{ for } c \in \mathcal{O}_{\gamma, p^i} \text{ where } \gamma - \beta +j \leq i \leq \gamma -1.\\\end{eqnarray*}

\end{itemize}
\end{itemize}

 \end{thm}
 
Proof of the above Theorem 3 is relegated to the Appendix.\\
 
We conclude this section with the following open problem.\\

\begin{prob} Let $G_{\lambda, p} = \mathbb{Z}/p^{\lambda_1}\mathbb{Z} \oplus \mathbb{Z}/p^{\lambda_2}\mathbb{Z} \oplus \cdots \oplus \mathbb{Z}/p^{\lambda_r}\mathbb{Z}$ be a finite $p$-group of rank $r$. For $r\geq 4$, classify $(a_1, a_2, \cdots, a_r)$-annihilators of $G_{\lambda, p}$.
\end{prob}

\section{Group-annihilator graph as a threshold graph}
Threshold graphs play an important role in graph theory as well as in several applied areas which include psychology and computer science \cite{MP}. These graphs were introduced by Chv\'{a}tal and Hammer  and Henderson and Zalcstein [10] in 1977. These graphs have been rediscovered in different contexts and therefore leading to several equivalent definitions. In this section, our main objective is to show that the group-annihilator graph relaised by a group $\mathbb{Z}/p^{\alpha}\mathbb{Z}$ is a threshold graph. 

A vertex in a graph $\mathcal{G}$ is called \textit{{\bf dominating}} if it is adjacent to every other vertex of $\mathcal{G}$. A graph $\mathcal{G}$ is called a \textit{{\bf threshold graph}} if it is obtained by the following procedure:\\
Start with $K_1$, a single vertex, and use any of the following steps, in any order, an arbitrary number of times:\\
(i) Add an isolated vertex.\\
(ii) Add a dominating vertex, that is, add a new vertex and make it adjacent to each existing vertex.

An \textit{{\bf alternating 4-cycle}} of a graph $\mathcal{G} = (V, E)$ is a configuration consisting of distinct vertices $a, b, c, d$ such that $(a,b), (c,d) \in E$  and $(a,c), (b,d) \notin E$. By considering the presence or absence of edges $(a,d)$ and $(b,c)$, we see that the vertices of an alternating $4$-cycle induce a path $P_4$, a square $C_4$, or a matching $2K_2$.

Threshold graphs can be characterised in many different ways. One of the characterisation of a threshold graph ~\cite[Theorem 1.2.4]{MP} is presented in the following result.

\begin{thm}
For a graph $\mathcal{G} = (V, E ),$ the following are equivalent:\\
1. $\mathcal{G}$ is a threshold graph;\\
2. $\mathcal{G}$ does not have an alternating $4$-cycle.
\end{thm}
In the subsequent result, we prove a very interesting property of a group-annihilator graph realised by the group $\Z/p^{\alpha}\Z$.

Recall that an \textit{independent part} (independent set) in a graph $\mathcal{G}$ is a set of vertices of $\mathcal{G}$ such that for very two vertices, there is no edge in the graph connecting the two. Also, the \textit{complete part} (complete subgraph) in a graph $\mathcal{G}$ is a set of vertices in $\mathcal{G}$ such that there is an edge between every pair of vertices.

\begin{thm}
For each integer $\alpha > 0$, the group annihilating graph $\Gamma(\Z/p^{\alpha}\Z)$ is a connected threshold graph. 
\end{thm}

\begin{proof}
We have, $$V(\Gamma(\Z/p^{\alpha}\Z)) = \Z/p^{\alpha}\Z =   (\dot\bigcup_{i=0}^{\alpha } \mathcal{O}_{p^i}), $$ where $\mathcal{O}_{\alpha, p^i}$'s are the orbits of the element $p^i$ for $0 \leq i\leq \alpha -1$ and $\mathcal{O}_{\alpha,p^{\alpha}} = \{0\}$, since $(a,0) \in E(\Gamma(\Z/p^{\alpha}\Z))$ for each $a \in V(\Gamma(\Z/p^{\alpha}\Z))$, therefore the graph $\Gamma(\Z/p^{\alpha}\Z)$ is a connected graph for every $\alpha >0.$ So we need to prove that $\Gamma(\Z/p^{\alpha}\Z)$ is a threshold graph. In other word, it is enough to prove that there is no alternating $4 - cycle$ in $\Gamma(\Z/p^{\alpha}\Z).$

For groups $\Z/3\Z$ and $\Z/2\Z$,  $|V(\Gamma(\Z/p^{\alpha}\Z))|< 4 $ so the theorem is vacuously true.
Consider $\Gamma(\Z/p^{\alpha}\Z)$ for which $|V(\Gamma(\Z/p^{\alpha}\Z))| \geq 4 .$
\begin{itemize}

\item{{\bf When $\alpha $ is even i.e. $\alpha=2k$ for some $k \geq 1$}}\\

It is clear from definition of the group-annihilator graph that $X = \dot\bigcup_{i=0}^{k-1}\mathcal{O}_{\alpha, p^i}$ and $Y = \dot\bigcup_{j=0}^{k-1}\mathcal{O}_{\alpha, p^{k+j}}$ are independent and complete part of the graph $\Gamma(\Z/p^{\alpha}\Z),$ where each element of both $X$ and $Y$ are connected to $\mathcal{O}_{\alpha,p^{\alpha}} =\{0\}.$ 

Let $a,b,c,d$ be four distinct vertices of the graph $\Gamma(\Z/p^{\alpha}\Z).$ 

{\bf Case - 1 : }If $a \in \mathcal{O}_{\alpha,p^{\alpha}} $ i.e. $a = 0$ and $b,c,d \in X \dot\cup  Y,$ then $(0,b),(0,c) \text{ and }(0,c)$ are in $E(\Gamma(\Z/p^{\alpha}\Z))$. Since there are no two disjoint edges, hence there is no alternating $4$-cycle.\\

{\bf Case - 2 :} Suppose that all $a,b,c,d \in X \cup Y.$\\
{\bf Subcase - 1: } If $a,b,c,d \in Y$ then every two disjoint edges have diagonal edges. So we are done.\\
{\bf Subcase - 2: } Let $a \in X$ and $b,c,d \in Y.$ Clearly $(b,c),(c,d),(d,b)\text{ are in } E(\Gamma(\Z/p^{\alpha}\Z)).$
If there is no edge from $a$ to any of $b,c,d$ then there are no disjoint edges. Therefore, there is no alternating $4$-cycle.\\
Assume that there is an edge from $a$ to either of $b,c,d.$ Without loss of generality assume that $(a,b)$ is an edge. Then $(a,b)\text{ and }(c,d)$ are the two disjoint edges but there is a diagonal edge $(b,c)\in E(\Gamma(\Z/p^{\alpha}\Z))$ adjacent to  $(a,b)\text{ and }(c,d)$ as shown in figure 2 below.\\

 \begin{align*}
\begin{pgfpicture}{6.3cm}{5.8cm}{8cm}{3cm}
\pgfnodecircle{Node1}[fill]{\pgfxy(4,7)}{0.1cm}
\pgfnodecircle{Node2}[fill]{\pgfxy(8,7)}{0.1cm}
\pgfnodecircle{Node3}[fill]{\pgfxy(7, 5)}{0.1cm}
\pgfnodecircle{Node4}[fill]{\pgfxy(9,5)}{0.1cm}
\pgfnodeconnline{Node1}{Node2}
\pgfnodeconnline{Node2}{Node3}
\pgfnodeconnline{Node3}{Node4}
\pgfnodeconnline{Node2}{Node4}
\pgfputat{\pgfxy(3.8, 7.3)}{\pgfbox[left,center]{$a$}}
\pgfputat{\pgfxy(7.9, 7.3)}{\pgfbox[left,center]{$b$}}
\pgfputat{\pgfxy(6.7, 5.3)}{\pgfbox[left,center]{$d$}}
\pgfputat{h\pgfxy(9.1, 5.3)}{\pgfbox[left,center]{$c$}}
\end{pgfpicture}
\end{align*}

\vspace{.5cm}

$\hskip .9cm \hskip .9cm \hskip .9cm \hskip .9cm \hskip .9cm \hskip .9cm \hskip .9cm \hskip .9cm \hskip .7cm Figure ~2$

Therefore, there is no alternating $4$-cycle.\\

{\bf Subcase - 3: } Let $a,b \in X$ and $c,d \in Y.$ Assume $a \in \mathcal{O}_{\alpha,p^{i_1}}$, $b \in \mathcal{O}_{\alpha, p^{i_2}}$ for $0 \leq i_1 \leq i_2 \leq k-1$
and $c \in \mathcal{O}_{\alpha, p^{k+j_1}}$, $d \in \mathcal{O}_{\alpha, p^{k+j_2}}$ for $0 \leq j_1 \leq j_2 \leq k-1.$ Therefore, $(a,b) \notin E(\Gamma(\Z/p^{\alpha}\Z))$ and 
$(c,d)\in E(\Gamma(\Z/p^{\alpha}\Z)).$\\
If there is no edges from $a,b$ to any of the vertices $c,d$, then we are done.

If $(a,c) \in E(\Gamma(\Z/p^{\alpha}\Z))$ then by the construction $(a,d),(b,c),(b,d) \in E(\Gamma(\Z/p^{\alpha}\Z)).$ In this case also diagonal edges are adjacent to any two disjoint edges. Therefore the result follows.\\

If $(b,c) \in E(\Gamma(\Z/p^{\alpha}\Z))$ then by the construction $(b,d)
\in E(\Gamma(\Z/p^{\alpha}\Z)).$ In this case, there is an edge from $a$ to any of the vertices $c,d$, again the result follows as above.\\

{\bf Subcase - 4:} Let $a,b,c \in X$ and $d \in Y.$ Clearly there is no edges in between $a,b,c. $ So  irrespective of the fact that there is an edge from any of the vertices $a,b \text{ or } c$ to $d,$ there are no disjoint edges contained in the graph, so there is no alternating 4-cycle.\\

{\bf Subcase - 5:} If $a,b,c,d$ are all belong to $V(\Gamma(\Z/p^{\alpha}\Z)),$ then there is no edge between any two vertices of $X.$ So, there is no alternating $4$-cycle.\\

\item{{\bf When $\alpha $ is odd i.e. $\alpha=2k+1$ for some $k \geq 0$}}\\

Let $X = \dot\bigcup_{i=0}^{k}\mathcal{O}_{\alpha, p^i}$ and $Y = \dot \bigcup_{j=1}^{k}\mathcal{O}_{\alpha, p^{k+j}}$ be independent and complete part of the graph $\Gamma(\Z/p^{\alpha}\Z),$ where each element of both the parts $X$ and $Y$ are connected to $\mathcal{O}_{\alpha, p^{\alpha}} =\{0\}.$

Now by considering four distinct vertices $a,b,c,d$ and going by the same similar argument as above for $\alpha$-even, we can easily prove that $\Gamma(\Z/p^{\alpha}\Z)$ has no alternating $4$-cycle.
\end{itemize}

This completes the proof. 
 
\end{proof}
\subsection{Creation sequence of $\Gamma(\Z/p^{\alpha}\Z)$ and its eigen values}

We may represent a threshold graph on $n$ vertices using a binary sequence $(b_1, b_2, \cdots, b_n)$, where $b_i = 0$ if vertex $v_i$ is being added as an isolated vertex and $b_i =  1$ if $v_i$ is being added as a dominating vertex. In constructing an adjacency matrix, we order the vertices in the same way they are given in their creation sequence.

For $\alpha$ odd or even, we consider the group-annihilator graph $\Gamma(\mathbb{Z}/p^{\alpha}\mathbb{Z})$ and determine its creation sequence. Also we determine the multiplicities of eigen values $0$ and $-1.$

Let $G= \Z/p^{\alpha}\Z$ be a cyclic group of order $p^{\alpha}.$ Then,
$$\Z/p^{\alpha}\Z = \dot\bigcup_{i=0}^{i=\alpha-1}\mathcal{O}_{\alpha,p^i}$$
{\bf Case -I $\alpha = 2k$, $k\geq 1$ :} The subgraph associated with each $\mathcal{O}_{\alpha,p^i}$ for $k\leq i \leq 2k-1$ is a complete graph and the subgraph associated with each $\mathcal{O}_{\alpha,p^i}$ for $1\leq i \leq k-1$ is an independent graph. 

Therefore the creation sequence for $\Gamma(\Z/p^{\alpha}\Z)$ in this case is the following:
$$01^{|\mathcal{O}_{\alpha,p^k}|-1}0^{|\mathcal{O}_{\alpha,p^{k-1}}|} 1^{|\mathcal{O}_{\alpha,p^{k+1}}|}\ldots0^{|\mathcal{O}_{\alpha,p}|} 1^{|\mathcal{O}_{\alpha,p^{2k-1}}|} 0^{|\mathcal{O}_{\alpha,1}|} 1^{|\mathcal{O}_{\alpha,p^{2k}}|}.$$

Therefore, the multiplicity of eigen value $0$ is
$$
(|\mathcal{O}_{\alpha,p^{k-1}}|-1)+(|\mathcal{O}_{\alpha,p^{k-2}}|-1)+\ldots+(|\mathcal{O}_{\alpha,1}|-1)$$
$$=\sum_{i=0}^{k-1}\frac{p^{\alpha}}{p^i}(1-\frac{1}{p}) -k= p^{\alpha} - p^k -k,$$
and multiplicity of eigen value $-1$ is
$$ 1+
(|\mathcal{O}_{\alpha,p^{k}}|-2)+(|\mathcal{O}_{\alpha,p^{k+1}}|-1)+\ldots+(|\mathcal{O}_{\alpha,p^{2k-1}}|-1)
+(|\mathcal{O}_{\alpha,p^{2k}}|-1)$$
$$= \sum_{i=0}^{k-1}(|\mathcal{O}_{\alpha,p^{k+i}}|-1)
= \sum_{i=0}^{k-1} \frac{p^{\alpha}}{p^{k+i}}(1-\frac{1}{p})-k = p^k -k -1.$$

There will be remaining $p^{\alpha}-(p^{\alpha} - p^k -k) -(p^k -k -1) = 2k +1$ many eigen values.
Now consider the equitable partition matrix $\mathcal{M}$
in the following manner:\\
Let $V_i = \mathcal{O}_{\alpha, p^{\alpha+1-i}}$ for $1\leq i \leq 2k+1.$ Then equitable partition matrix $\mathcal{M} =(m_{ij}),$ where $$m_{ij}= |N(u)\cap V_j| \text{ for } u \in V_i,$$ where $N(u)$ is the number of adjacent vertices to $u.$ These remaining $2k+1$ eigen values are the roots of the characteristic polynomial of $\mathcal{M}.$

{\bf Case -II $\alpha = 2k-1$, $k\geq 1$ :} The subgraph associated with each $\mathcal{O}_{\alpha,p^i}$ for $k+1\leq i \leq 2k-1$ is a complete graph and the subgraph associated with each $\mathcal{O}_{\alpha,p^i}$ for $1\leq i \leq k$ is an independent graph. 

Therefore the creation sequence for $\Gamma(\Z/p^{\alpha}\Z)$ in this case is the following:
$$0^{|\mathcal{O}_{\alpha,p^{k-1}}|}1^{|\mathcal{O}_{\alpha,p^{k}}|} 0^{|\mathcal{O}_{\alpha,p^{k-2}}|}
1^{|\mathcal{O}_{\alpha,p^{k+1}}|} 
\ldots0^{|\mathcal{O}_{\alpha,p}|} 1^{|\mathcal{O}_{\alpha,p^{2k-2}}|} 0^{|\mathcal{O}_{\alpha,1}|} 1^{|\mathcal{O}_{\alpha,p^{2k-1}}|}.$$

Consider $p=2 \text{ and } k=1$ then creation sequence will become $01,$ therefore multiplicity of eigen values $0$ is $0$ and $-1$ is $1.$

In other cases, the multiplicity of eigen value $0$ is
$$
(|\mathcal{O}_{\alpha,p^{k-1}}|-1)+(|\mathcal{O}_{\alpha,p^{k-2}}|-1)+\ldots+(|\mathcal{O}_{\alpha,1}|-1)$$
$$=\sum_{i=0}^{k-1}\frac{p^{\alpha}}{p^i}(1-\frac{1}{p}) -k= p^{\alpha} - p^{k-1} -k,$$

and multiplicity of eigen value $-1$ is
$$ 
(|\mathcal{O}_{\alpha,p^{k}}|-1)+(|\mathcal{O}_{\alpha,p^{k+1}}|-1)+\ldots+(|\mathcal{O}_{\alpha,p^{2k-2}}|-1)
+(|\mathcal{O}_{\alpha,p^{2k-1}}|-1)$$
$$= \sum_{i=0}^{k-2}(|\mathcal{O}_{\alpha,p^{k+i}}|-1)
= \sum_{i=0}^{k-2} \frac{p^{\alpha}}{p^{k+i}}(1-\frac{1}{p})-k +1= p^{k-1} -k.$$

There will be remaining $p^{\alpha}-(p^{\alpha} - p^{k-1} -k) -(p^{k-1} -k ) = 2k $ many eigen values. Now consider the equitable partition matrix $\mathcal{M}$
in the following manner:\\
Let $V_i = \mathcal{O}_{\alpha, p^{\alpha+1-i}}$ for $1\leq i \leq 2k.$ Then equitable partition matrix $\mathcal{M} =(m_{ij}),$ where $$m_{ij}= |N(u)\cap V_j| \text{ for } u \in V_i,$$ where $N(u)$ is the number of adjacent vertices to $u.$ These remaining $2k+1$ eigen values are the roots of the characteristic polynomial of $\mathcal{M}.$

\subsection{Adjacency spectrum and energy of a group-annihilator graph}
\defn Let $\mathcal{G}$ be a graph with vertex set $V(\mathcal{G}) = \{v_1,\cdots,v_n\}$ and edge set $E(\mathcal{G})$. The adjacency matrix of $\mathcal{G},$ denoted by $\mathcal{A}(\mathcal{G})$, is the $n \times n$ matrix with entries $a_{ij} $ defined as follows: 
\begin{eqnarray*}
a_{ij}&=& 0 \text{ if } i=j,\\
&=& 0 \text{ if } i\neq j \text{ and } v_i \text{ is not adjacent to } v_j,\\
&=& 1 \text { if } i\neq j \text{ and } v_i \text{ is adjacent to } v_j.\\
\end{eqnarray*}
We often denote $\mathcal{A}(\mathcal{G})$ as $\mathcal{A}.$\\
Consider the cyclic group $G=\Z/p^r \Z$ of order $p^r$,
where $p$ is a prime and $r\in \mathbb{N}.$ Let $\Gamma(G)$ be the group-annihilator graph of $G.$ The vertex set of $\Gamma (G)$ is $V(\Gamma(G)) = \{1,\cdots,p^r\}$. If the rows and columns of an $p^r \times p^r$ matrix are indexed by $V(\Gamma(G)).$ Then the adjacency matrix of $\Gamma(G)$, denoted by  $\mathcal{A}(\Gamma(G))$, is the matrix with entries $a_{ij} $ defined as follows: 
\begin{eqnarray*}
a_{ij}&=& 0 \text{ if } i=j,\\
&=& 0 \text{ if } i\neq j \text{ and } i+j < r,\\
&=& 1 \text { if } i\neq j \text{ and } i+j \geq r.\\
\end{eqnarray*}

\begin{exm} Let $G = \mathbb{Z}/8\mathbb{Z}$. Then the adjacency matrix of group-annihilator graph realised by $G$ is following: 
\begin{equation*}
\mathcal{A}(\Gamma(G)) = 
\begin{pmatrix}
0 & 1 & 1 & 1 & 1 & 1 & 1 & 1 \\
1 & 0 & 0 & 0 & 0 & 0 & 0 & 0 \\
1 & 0 & 0 & 0 & 1 & 0 & 0 & 0 \\
1 & 0 & 0 & 0 & 0 & 0 & 0 & 0 \\
1 & 0 & 1 & 0 & 0 & 0 & 1 & 0 \\
1 & 0 & 0 & 0 & 0 & 0 & 0 & 0 \\
1 & 0 & 0 & 0 & 1 & 0 & 0 & 0 \\
1 & 0 & 0 & 0 & 0 & 0 & 0 & 0 \\
\end{pmatrix}
\end{equation*}
\end{exm}
The adjacency matrix $\mathcal{A}(\Gamma(G))$ of group-annihilator graph realised by $G=\Z/p^r \Z$ is a symmetric matrix for every $r\in \mathbb{N}. $ The set of solution of the characteristic polynomial $det (\lambda \mathcal{I}_{p^r} - \mathcal{A}(\Gamma(G))) = 0$ in $\mathbb{C}$ is known as adjacency spectrum of the group-annihilator graph. For $r=1,$ the group-annihilator graph is  {\it star graph}, so its adjacency spectrum is $\{\sqrt{p-1}, \underbrace{0,\ldots,0}_{p-2\textrm{\ times}},-\sqrt{p-1} \}.$ Moreover, for the $p$-group  $G^{*}_{\lambda, p} = \mathbb{Z}/p^{\lambda_1}\mathbb{Z} \oplus \mathbb{Z}/p^{\lambda_2}\mathbb{Z} \oplus \cdots \oplus \mathbb{Z}/p^{\lambda_r}\mathbb{Z}$, where $\lambda = \lambda_1 = \lambda_2 \cdots = \lambda_r$, the group-annihilator graph realised by $G^{*}_{\lambda, p}$ is a complete graph on $p^{r\lambda}$ vertices. The adjacency matrix $\mathcal{A}(\Gamma(G^{*}_{\lambda, p})) = J - I$, where $J$ is a matrix of order $p^{r\lambda} \times p^{r\lambda}$ where each entry is $1$ and $I$ is an identity matrix of order $p^{r\lambda} \times p^{r\lambda}$. The adjacency spectrum of $\mathcal{A}(\Gamma(G^{*}_{\lambda, p}))$ is $\{\underbrace{-1, -1, \cdots, -1}_{p^{r\lambda}-1\textrm{\ times}}, p^{r\lambda}-1\}$.  

Let $\mathcal{G}$ be a graph  with $n$ vertices and $\lambda_1, \cdots, \lambda_n$ be the $n$ eigen values of it. Then the {\bf energy} of $\mathcal{G}$ is defined as;
  $$E(\mathcal{G}) = \sum_{i=1}^n |\lambda_i|.$$
The energy of a group-annihilator graph realised by  $G=\Z/p \Z$ is $2\sqrt{p-1}$, whereas the energy of $\Gamma(G^{*}_{\lambda, p})$ is $2(p^{r\lambda} - 1)$. 
  
\defn: A graph $\mathcal{G}$ on $n$ vertices is said to be {\bf hyperenergetic graph} if $E(\mathcal{G}) > 2 (n-1) = E(K_n) =  $energy of a complete graph with n vertices. $\mathcal{G}$ is said to be {\bf hypoenergetic graph} if $E(\mathcal{G}) < n.$ 

In \cite{JTT}, authors have asserted that it is still of interest to find class of hyperenergtic graphs. In fact, they have shown that for $n > 8$ the threshold graphs $01^{\floor{\frac{n}{2}}-2}01^{\ceil{\frac{n}{2}}}$ are hyperenergetic. Hyperenergeticity was verified for other varieties of various classes of graphs: Paley, circulant, Kneser etc. The major blow to the research of hyperenergetic graphs was given by V. Nikiforv \cite{N} who proved the following very interesting result regarding hyperenergetic graphs.\\

\begin{thm}
For almost all graphs $\mathcal{G}$,

\begin{center} $E(\mathcal{G}) = \left(\frac{4}{3\pi} + o(1)\right)n^{\frac{3}{2}}$.
\end{center}
\end{thm}

The preceeding theorem immediately implies that almost all graphs are hyperenergetic. So, making further search for hyperenergetic graphs is pointless. However, it is interesting to find class of threshold graphs $\mathcal{G'}$ with $n$ vertices
for which the inequality $E(\mathcal{G}) < E(\mathcal{G^{'}}) < E(K_n)$ holds, where $\mathcal{G}$ is a graph with $n$ vertices.\\

\begin{thm}
For a prime $p \geq 7,$ there exists a connected threshold graph $\mathcal{G}$ with $p^2$ vertices such that $$ E(\Gamma(\Z/p^2\Z)) < E(\mathcal{G}) < E(K_{p^2}).$$
 \end{thm}
 
 \begin{proof}
The creation sequence of a group-annihilator graph realised by the group $\Z/p^2\Z$ is $01^{p-2}0^{p^2-p}1.$
Therefore, the multiplicity of eigen value $0$ is $p^2-p-1$ and the multiplicity of eigen value $-1$ is $p-2.$ There will be remaining $p^2 - (p^2-p-1)- (p-2)= 3$ eigen values. These three eigen values are the roots of the following polynomial, 
 $$f(x) =x^3 - (p-2)x^2 - (p^2-1)x + p(p-1)(p-2),$$
whereas the characteristic polynomial of a group-annihilator graph corresponding to $\Z/p^2\Z$ is  $x^{p^2-p-1}(x+1)^{p-2}f(x)$.
 
Now, $$f(0) = p (p-1)(p-2) > 0 \text { , }f(p)= - p (p-3) < 0, $$
 $$f(-p) = -p (p-1) \text{ and } f(2p)= p (3p^2 +5p +4) > 0.$$
If $\lambda_i$ for $i=1,2,3$ are roots of the polynomial $f(x)$ with $\lambda_1< \lambda_2 < \lambda_3$, then by Intermediate Value Property $|\lambda_i|\leq 2p.$ Therefore, $$E(\Gamma(\Z/p^2\Z))=(p-2) +
 \sum_{i=1}^3 |\lambda_i| \leq (p-2)+ 6p = 7p-2.$$ 
Let $\mathcal{G}$ be a connected threshold graph with creation sequence $ 01^{2m}0^{4m}1^{2m},$ where $8m+1= p^2.$\\
Therefore, the multiplicity of eigen value $0$ is $4m-1$ and the multiplicity of eigen value $-1$ is $4m-1.$ There will be remaining $8m+1 - 2(4m-1)= 3$ eigen values. These three eigen values are the roots of the following polynomial 
 $$g(x) =x^3 - (4m-1)x^2 - 4m(2m+1)x + 16m^3,$$
whereas the characteristic polynomial of graph $\mathcal{G}$ is  $x^{4m-1}(x+1)^{4m-1}g(x)$.
 
 Now, $$g(0) = 16m^3 > 0 \text { , } g(-\frac{5m}{2})= -\frac{m^2}{8}(37m -130) < 0, $$
 $$g(\frac{3m}{2}) = -\frac{m^2}{8}(13m+30) < 0\text{ and } g(\frac{13m}{2})=
  \frac{m^2}{8}(557m+130) > 0.$$
  If $\lambda_i'$ for $i=1,2,3$ roots of the polynomial $g(x)$ with $\lambda_1'< \lambda_2' < \lambda_3'$ then by Intermediate Value Property $|\lambda_1'|\leq \frac{5m}{2},|\lambda_2'|\leq \frac{3m}{2}, |\lambda_3'|\leq \frac{13m}{2} .$ Therefore, $$E(\Gamma(\Z/p^2\Z))\leq 7p-2
  < 4m-1 = \frac{p^2-1}{2}< (4m-1)+ \sum_{i=1}^3|\lambda_i'| = E(\mathcal{G}).$$ 
  Therefore, $$E(\mathcal{G})\leq (4m-1) +\frac{5m}{2}+\frac{3m}{2}+\frac{13m}{2} = (4m-1)+\frac{21m}{2}=14m+\frac{m}{2}-1< 16m = 2 (p^2-1)= E(K_{p^2}).$$

This completes the proof.

\end{proof}

It follows from the preceeding theorem that a group-annihilator graph realised by $\Z/p^2\Z$ is not hyperenergetic but hypoenergetic. We have the following conjecture regarding the class of hypoenergetic threshold graphs.

\begin{conj}
For any positive integer $\alpha,$ the group-annihilator graph realised by the group $\Z/p^{\alpha}\Z$ is not hyperenergetic but hypoenergetic.
 \end{conj}
 
\section{Laplacian eigen values of  $\bf {\Gamma (\mathbb{Z}/p^{\alpha}\mathbb{Z})}$}

The Laplacian matrix of a graph $\mathcal{G}$ is $L(\mathcal{G}) = D(\mathcal{G}) - \mathcal{A}(\mathcal{G})$, where $\mathcal{A}(\mathcal{G})$ is the adjacency matrix of $\mathcal{G}$ and $D(\mathcal{G})$ is the diagonal matrix with diagonal entries as vertex degrees of $\mathcal{G}$. Laplacian matrices are well-studied in the field of spectral graph theory. See \cite{FC, RM} for more information on the Laplacian matrix. The eigen values of $L(\mathcal{G})$ are known as Laplacian eigen values. The main purpose of this section is to study Laplacian eigen values of the graph $\mathcal{G}$ realised by the group $G = \mathbb{Z}/p^{\alpha}\mathbb{Z}$. We use the concept of conjugate sequence of non-negative integers to determine that Laplacian eigen values of the graph $\Gamma(\mathbb{Z}/p^{\alpha}\mathbb{Z})$ are representatives $0, 1, p, p^2, \cdots, p^{\alpha - 1}$ (with multiplicities) of orbits $\{\mathcal{O}_{\alpha, p^{\alpha}}\} \cup \{\mathcal{O}_{\alpha, p^{i}} : 0\leq i \leq \alpha - 1\}$.

\begin{defn} The \textit{degree sequence} of a graph $\mathcal{G}$ is given by $\pi(\mathcal{G}) = (d_1, d_2, \cdots, d_n)$, which is the non-increasing sequence of non-zero degrees of vertices of $\mathcal{G}$.
\end{defn}
The degree sequence is a graph invariant, so two isomorphic graphs have the same degree sequence. In general, the degree sequence does not uniquely determine a graph, that is, two non-isomorphic
graphs can have the same degree sequence. However, for threshold graphs, we have the following result.

\begin{prop}\cite{RMB} Let $\mathcal{G}_1$ and $\mathcal{G}_2$ be two threshold graphs and let $\pi_1(\mathcal{G}_1)$ and $\pi_{2}(\mathcal{G}_2)$ be the degree sequences of $\mathcal{G}_1$ and $\mathcal{G}_2$, respectively. If $\pi_1(\mathcal{G}_1) = \pi_{2}(\mathcal{G}_2)$, then $\mathcal{G}_1 \cong \mathcal{G}_2$.
\end{prop}
The Laplacian spectra of threshold graphs have been studied in \cite{HK, RB}. In \cite{HK}, the formulas for the Laplacian spectrum, the Laplacian polynomial, and the number of spanning trees of a threshold graph are given. It is shown that the degree sequence of a threshold graph and the sequence of eigenvalues of its Laplacian matrix are \enquote {almost the same} and on this basis, formulas are given to express the Laplacian polynomial and the number of spanning trees of a threshold graph in terms of its degree sequence.

Let $x = [x_1, x_2, \cdots, x_n]$ be a sequence of non-negative integers arranged in non-increasing order, which we refer to as a \textit{partition}. Define the transpose of the partition as $x^* = [x_1^{*}, x_2^{*}, \cdots, x_m^{*}]$, where $x_j^{*} = |\{x_i : x_i \geq j\}|$, $j = 1, 2,\cdots, m$. Therefore $x_j^{*}$ is the number of $x_i$'s that are greater than equal to $j$. Recall from \cite{RB} that a sequence $x^*$ is called the conjugate sequence of $x$. The another interpretation of a conjugate sequence is the \textit{Ferrer's diagram} corresponding to $x_1, x_2, \cdots, x_n$ consists of $n$ left justified rows of boxes, where the $i^{th}$ row consists of $x_i$ boxes (blocks), $i = 1, 2, \cdots, n$. Note that $x_i^{*}$ is the number of boxes in the $i^{th}$ column of the Ferrer's diagram with $i = 1, 2, \cdots, n$. An immediate consequence of this observation is that if $x^*$ is the conjugate sequence of $x$, then;
\begin{equation*}
\sum\limits_{i = 1}^{n} x_i = \sum\limits_{i = 1}^{m} x_i^{*}
\end{equation*}

If $x$ represents the degree sequence of a graph, then the number of boxes in the $i^{th}$ row of the Ferrers diagram is the degree of vertex $i$, while the number of boxes in the $i^{th}$ row of the Ferrers diagram of the transpose is the number of vertices with degree at least $i$. The trace of a partition $x$ is $tr(x) = |\{i : x_i\geq i\}|$, which is the length of \enquote {diagonal} of the Ferrer's diagram for $x$ (or $x^*$).

In the following result, we show that Laplacian eigen values of the group-annihilator graph $\Gamma(G)$ enjoy an interesting property.

\begin{thm} \cite{RB}
Let $\mathcal{G}$ be a threshold graph with $V(\mathcal{G}) = \{1, 2, \cdots, n\}$, $L(\mathcal{G})$ be the Laplacian and $[d_1, d_2, \cdots, d_n]$ the degree sequence of $\mathcal{G}$. Then $[d_1^*, d_2^*, \cdots, d_n^*]$ are the eigenvalues of $L(\mathcal{G})$.
\end{thm}

\begin{thm}\label{thm lap} Let $p$ be a prime number and $\alpha \in \mathbb{N}.$
Let $\Gamma(G)$ be a group-annihilator graph realised by the group $G = \mathbb{Z}/p^{\alpha}\mathbb{Z}$ and let $L(\Gamma(G))$ be the Laplacian. Then $[p^{\alpha}, p^{\alpha - 1}, \cdots, 1]$ are the eigen values of $L(\Gamma(G))$.
\end{thm}

\begin{proof} We consider the following cases:

\textbf{Case - I : $\alpha = 2k$ for $k\ge 1:$}

For the symmetric action, $Aut(\Gamma(G)) \times G \longrightarrow G$,

the number of distinct elements in the orbit represented by an element $1$ are $$\phi(p^{\alpha}) = p^{\alpha} - p^{\alpha - 1} = |\mathcal{O}_{\alpha,1}|,$$ 

the number of distinct elements in the orbit represented by an element $p$ are $${\frac{\phi(p^{\alpha})}{p}} = p^{\alpha - 1} - p^{\alpha - 2} = |\mathcal{O}_{\alpha,p}|,$$
\hskip .9cm \hskip .9cm \hskip .9cm \hskip .9cm \hskip .9cm\hskip .9cm \hskip .9cm \hskip .9cm \vdots

the number of distinct elements in the orbit represented by an element $p^{\alpha - 1}$ are $${\frac{\phi(p^{\alpha})}{p^{\alpha - 1}}} = p - 1 =  |\mathcal{O}_{\alpha,p^{\alpha - 1}}|.$$

By definition, vertices of the graph $\Gamma(G)$ are labelled by elements of the group $G$, so there is a unique map, $\gamma: V(\Gamma(G)) \longrightarrow G$, such that for some $v \in V(\Gamma(G))$,
\begin{center}
$\gamma(v) = 0$.
\end{center}
That is, there is a unique vertex in $\Gamma(G)$ labelled as $0$ and degree as $p^{\alpha} - 1$.

Moreover, there are $\phi(p^{\alpha})$ vertices in $\Gamma(G)$ with degree equal to $1$, since these vertices are adjacent to vertex $0$ only. The number of distinct elements in the orbit represented by $p^{\alpha - 1}$ are $p - 1$, so the degree of vertices of $\Gamma(G)$ contained in the orbit $\mathcal{O}_{\alpha,p}$ is $1 + p - 1 = p$. Similary the degree of vertices of $\Gamma(G)$ contained in the orbit $\mathcal{O}_{\alpha,p^2}$ is $p^2$. We partition the vertices of $\Gamma(G)$ as follows;
\begin{center}
$\bigcup\limits_{i = 1}^{\alpha}\mathcal{O}_{\alpha,p^{i}} = (\bigcup\limits_{i = 1}^{k-1}\mathcal{O}_{\alpha,p^{i}}) \dot{\bigcup} (\bigcup\limits_{i = k }^{\alpha}\mathcal{O}_{\alpha,p^{i}})$,
\end{center}
where $\bigcup\limits_{i = 1}^{k-1}\mathcal{O}_{\alpha,p^{i}}$ represents an independent part of $\Gamma(G)$, that is the set of vertices of $\Gamma(G)$ which are not adjacent to each other and $\bigcup\limits_{i = k }^{\alpha}\mathcal{O}_{\alpha,p^{i}}$ represents a complete part of $\Gamma(G)$, that is the set of vertices of $\Gamma(G)$ which are mutually connected to each other. Thus for any $a\in \mathcal{O}_{\alpha,p^{i}}$, where $1\leq i \leq k-1$, we have,
\begin{eqnarray*}
deg(a) &=& |\mathcal{O}_{\alpha,p^{\alpha}}| + |\mathcal{O}_{\alpha,p^{\alpha - i}}| + \cdots + |\mathcal{O}_{\alpha,p^{\alpha-1}}|\\
&=& 1 + {\frac{\phi(p^{\alpha})}{p^{\alpha-i}}} + {\frac{\phi(p^{\alpha})}{p^{\alpha-i+1}}} + \cdots + {\frac{\phi(p^{\alpha})}{p^{\alpha-1}}}\\
&=& 1 + p^{i-1}(p-1) + p^{i-2}(p-1) + \cdots + p^{2} - p + p - 1\\
&=& p^{i} 
\end{eqnarray*}
On the other hand, if $a\in \mathcal{O}_{\alpha,p^{i}}$, where $k\leq i \leq \alpha$, then,
\begin{eqnarray*}
deg(a) &=& |\mathcal{O}_{\alpha,p^{\alpha}}| + |\mathcal{O}_{\alpha,p^{\alpha - i}}| + \cdots + |\mathcal{O}_{\alpha,p^{\alpha-1}}| -1\\
&=& 1 + {\frac{\phi(p^{\alpha})}{p^{\alpha-i}}} + {\frac{\phi(p^{\alpha})}{p^{\alpha-i+1}}} + \cdots + {\frac{\phi(p^{\alpha})}{p^{\alpha-1}}} - 1\\
&=& p^{i-1}(p-1) + p^{i-2}(p-1) + \cdots + p^{2} - p + p - 1\\
&=& p^{i} - 1.
\end{eqnarray*}  
Thus for the independent part degree sequence of the graph $\Gamma(G)$ is,
\begin{equation*}
\pi_{1}(\Gamma(G)) = \{\underbrace{1, 1, \ldots, 1}_{\phi(p^{\alpha})\textrm{\ times}}, \underbrace{p, p, \ldots, p}_{\frac{\phi(p^{\alpha})}{p}\textrm{\ times}}, \ldots, \underbrace{p^{k-1}, p^{k-1}, \ldots, p^{k-1}}_{\frac{\phi(p^{\alpha})}{p^{k-1}}\textrm{\ times}}\}.  
\end{equation*}
Also, for the complete part degree sequence of $\Gamma(G)$ is,
\begin{equation*}
\begin{aligned}
\pi_{2}(\Gamma(G)) {} & = \{\underbrace{p^{k} - 1, p^{k} - 1, \ldots, p^{k} - 1}_{\frac{\phi(p^{\alpha})}{p^{k}}\textrm{\ times}}, \underbrace{p^{k+1} - 1, p^{k+1} - 1, \ldots, p^{k+1} - 1}_{\frac{\phi(p^{\alpha})}{p^{k+1}}\textrm{\ times}}, \ldots,\\
&  \hskip .9cm \hskip .9cm \hskip .9cm \hskip .5cm \underbrace{p^{\alpha -1} -1, p^{\alpha -1} -1, \ldots, p^{\alpha-1}-1}_{\frac{\phi(p^{\alpha})}{p^{\alpha-1}}\textrm{\ times}}\}. 
\end{aligned} 
\end{equation*}
The degree sequence of the graph $\Gamma(G)$ is $\pi_{1}(\Gamma(G)) \cup \pi_{2}(\Gamma(G)) \cup \{p^{\alpha}-1\}= \pi(\Gamma(G))$. 
The another representation of $\pi(\Gamma(G))$ is the Ferrer's diagram. It is easy to verify that in the Ferrer's diagram there are,

\hskip .9cm\hskip .9cm \hskip .9cm \hskip .9cm\hskip .9cm \hskip .9cm $\phi(p^{\alpha})$ rows of single block, 

\hskip .9cm\hskip .9cm \hskip .9cm \hskip .9cm\hskip .9cm \hskip .9cm ${\frac{\phi(p^{\alpha})}{p}}$ rows of $p$ blocks,

\hskip .9cm\hskip .9cm \hskip .9cm \hskip .9cm\hskip .9cm \hskip .9cm ${\frac{\phi(p^{\alpha})}{p^2}}$ rows of $p^2$ blocks,

\hskip .9cm \hskip .9cm \hskip .9cm \hskip .9cm \hskip .9cm \hskip .9cm \hskip .9cm \hskip .9cm \vdots

\hskip .9cm\hskip .9cm \hskip .9cm \hskip .9cm\hskip .9cm \hskip .9cm ${\frac{\phi(p^{\alpha})}{p^{\alpha-1}}}$ rows of  $p^{\alpha-1}$ blocks,

\hskip .9cm\hskip .9cm \hskip .9cm \hskip .9cm\hskip .9cm \hskip .9cm and a single row (top most row) of $p^{\alpha - 1}$ blocks.

The \textit{Ferrer's conjugate} is the transpose of the Ferrer's diagram of $\pi(\Gamma(G))$. In the Ferrer's conjugate we see that,

the number of blocks in the top most row are,
\begin{eqnarray*}
|\mathcal{O}_{\alpha,p^{\alpha}}| + |\mathcal{O}_{\alpha,p^{\alpha - 1}}| + \cdots + |\mathcal{O}_{\alpha,1}| &=& 1 + {\frac{\phi(p^{\alpha})}{p^{\alpha-1}}} + {\frac{\phi(p^{\alpha})}{p^{\alpha-2}}} + \cdots + {\phi(p^{\alpha})}\\
&=& 1 + p-1 + p^{2} - p + \cdots + p^{\alpha - 1} - p^{\alpha-2} + p^{\alpha} - p^{\alpha - 1}\\
&=& p^{\alpha}, 
\end{eqnarray*}
the number of blocks in the second top most row are,
\begin{eqnarray*}
|\mathcal{O}_{\alpha,p^{\alpha}}| + |\mathcal{O}_{\alpha,p^{\alpha - 1}}| + \cdots + |\mathcal{O}_{\alpha,p}| &=& 1 + {\frac{\phi(p^{\alpha})}{p^{\alpha-1}}} + {\frac{\phi(p^{\alpha})}{p^{\alpha-2}}} + \cdots + {\frac{\phi(p^{\alpha})}{p}}\\
&=& p^{\alpha - 1},
\end{eqnarray*}
 
\hskip .9cm\hskip .9cm \hskip .9cm \hskip .9cm \hskip .9cm\hskip .9cm \hskip .9cm \hskip .9cm  \vdots

the number of blocks in rows which are represented by an orbit $\mathcal{O}_{\alpha,p}$ are 
\begin{eqnarray*}
|\mathcal{O}_{\alpha,p^{\alpha}}| + |\mathcal{O}_{\alpha,p^{\alpha - 1}}| &=& 1 + {\frac{\phi(p^{\alpha})}{p^{\alpha-1}}}\\
&=& p.  
\end{eqnarray*}
Note that the number of rows that contain $p$ number of blocks is $|\mathcal{O}_{\alpha,p}|$, also the number of rows that consist of single block is $\phi(p^{\alpha})$ which is the cardinality of the orbit represented by $1$.

Thus, the conjuagte sequence of $\pi({\Gamma(G)})$ is,

\begin{equation*}
\pi^{*}(\Gamma(G)) = \{p^{\alpha}, p^{\alpha-1}, \ldots,  \underbrace{p^{k}, p^{k}, \ldots, p^{k}}_{\frac{\phi(p^{\alpha})}{p^k}\textrm{\ times}}, \ldots,  \underbrace{p, p, \ldots, p}_{\frac{\phi(p^{\alpha})}{p}\textrm{\ times}},  \underbrace{1, 1, \ldots, 1}_{\phi(p^{\alpha})\textrm{\ times}}\}.  
\end{equation*}

\textbf{Case-II : $\alpha =2k+1$ for $k\ge 0$}

One can prove this case in the same manner as above.

\end{proof}

\begin{cor} Let $\Gamma(G)$ be the group-annihilator graph relaised by $G = \mathbb{Z}/p^{\alpha}\mathbb{Z}$ and let $L(\Gamma(G))$ be the Laplacian. Then the representative $0, 1, p, p^2, \cdots, p^{\alpha - 1}$ (with multiplicities) of orbits $\{\mathcal{O}_{\alpha, p^{\alpha}}\} \cup \{\mathcal{O}_{\alpha, p^{i}} : 0\leq i \leq \alpha - 1\}$ are the eigen values of $L(\Gamma(G))$.
\end{cor}

\begin{rem} A graph is called \textit{Laplacian integral} if the eigenvalues of its Laplacian are all integers. Threshold graphs are Laplacian integral. In particular, the group-annihilator graph realised by $\mathbb{Z}/p^{\alpha}\mathbb{Z}$ is a Laplacian integral graph. Besides threshold graphs there are other
Laplacian integral graphs as well. It is not difficult to see that any regular graph that is adjacency integral (all eigen values of its adjacency matrix are integers) is necessarily Laplacian integral. 
\end{rem}

We conclude this section with the following example which illustrates Theorem \ref{thm lap}.

\begin{exm} Let $G = \mathbb{Z}/2^4\mathbb{Z}$ be a group. The degree sequence of $\Gamma(G)$ is,
\begin{equation*}
\pi(\Gamma(G)) = \{15, 7, 3, 3, 2, 2, 2, 2, 1, 1, 1, 1, 1, 1, 1, 1\}.
\end{equation*}
The conjugate sequence of $\pi(\Gamma(G))$ which is obtained by taking transpose of the Ferrer's diagram is,

\begin{equation*}
\pi^{*}(\Gamma(G)) = \{2^4, 2^3, 2^2, 2, 2, 2, 2, 1, 1, 1, 1, 1, 1, 1, 1\}.
\end{equation*}
Therefore by Theorem \ref{thm lap}, it follows that representatives $1, 2, 2^2, 2^3, 2^4$ of orbits $\mathcal{O}_{4, 1}, \mathcal{O}_{4, 2^1}, \mathcal{O}_{4, 2^2}, \mathcal{O}_{4, 2^3}$ and $\mathcal{O}_{4, 2^4}$ are the Laplacian eigen values of $\Gamma(G)$. Note that the multiplicity of $1$ is $8$ and the multiplicity of $2$ is $4$. 
\end{exm} 

{\bf Acknowledgement:}This research project was initiated when the second author visited the Stat-Math Unit, Indian Statistical Institute Bangaluru, India. So, we are immenesely grateful to ISI Bangaluru for all the facilities. Moreover, the authors would like to thank Amitava Bhattacharya of TIFR Mumbai for the motivation of  this research work. The first author's research is being supported by National Board of Higher Mathematics, Department of Atomic Energy, Govt. of India.

\section{Appendix: Proof of Theorem 3}

We furnish the proof of Theorem 3 in this section.

\begin{proof}

\begin{itemize}
\item Consider the case $a \in \mathcal{O}_{\alpha, p^{\alpha}}$,  and $b\in \mathcal{O}_{\beta, p^{\beta}}$
\begin{itemize}
\item For $c \in \mathcal{O}_{\gamma, p^{\gamma}},$ it is trivially true that $[(a,b,c):G]= p^{\gamma}\mathbb{Z}.$
\item Let $c \in \mathcal{O}_{\gamma, p^i}$, where $0\leq i \leq \beta -1.$ Clearly, $p^{\beta}\mathbb{Z} \subset [(a,b,c): G].$ Now by considering $y \in [(a,b,c):G]$, we see that $y (0 ,1, 0)= n (0,0,c)$ for some integer $n$. That completes the other part.
\item Let $c \in \mathcal{O}_{\gamma, p^i}$, where $\beta\leq i \leq \gamma -1.$ Then $c= p^i k'$, where $(k',p) = 1.$ Clearly, $p^{i}\mathbb{Z} \subset [(a,b,c): G].$ Now by considering $y \in [(a,b,c):G]$, we again see that $y (0 ,0, 1)= n (0,0,c)$ for some integer $n.$ So, we are done.
\end{itemize}

\item Consider the case $a \in \mathcal{O}_{\alpha, p^{\alpha}}$, so $a=0$  and $b\in \mathcal{O}_{\beta, p^{j}}$, where $0 \leq j \leq \beta -1.$ So, $b= p^jb'$ for some $b'$, where $(b',p)=1.$
\begin{itemize}
\item For $c \in \mathcal{O}_{\gamma, p^{\gamma}},$ it is trivially true that $[(a,b,c):G]= p^{\gamma}\mathbb{Z}.$
\item Let $c \in \mathcal{O}_{\gamma, p^i}$, where $0\leq i \leq j $.  So, $c= p^ic'$ for some $c'$, where $(c',p)=1.$ Then for any $y = p^{\beta}k \in p^{\beta}\mathbb{Z},$ $y (u,v,w) = (0, 0, p^\beta) kw = wkc'^{-1}p^{\beta -i} (a,b,c)$ as $j-i \geq 0.$ So $p^{\beta}\mathbb{Z} \subset [(a,b,c): G].$
Now by considering $y \in [(a,b,c):G]$, we have $y (0 ,1, 0)= n (0,0,c)$ for some integer $n,$ so we are done, since $j-i \geq 0.$
\item Let $c \in \mathcal{O}_{\gamma, p^i}$, where $ j+1\leq i \leq \gamma - \beta +j.$ Then $c= p^i c'$ for some $c'$, where $(c',p) = 1.$ For any $y = p^{i+\beta -j}k' \in p^{\beta}\mathbb{Z},$
$y (u,v,w) = (0, 0, p^{\beta+i-j}) k'w  = k'wc'^{-1}p^{\beta -j} (a,b,c)$ as $\beta - j \geq 1.$ 
So, $p^{\beta+i -j}\mathbb{Z} \subset [(a,b,c): G].$
Now by considering $y \in [(a,b,c):G]$, we have $y (0 ,0, 1)= n (0,b,c)$ for some integer $n,$ this implies $p^{\beta -j}\mid n$ and $y = p^{\beta+i-j}c' (\bmod p^{\gamma} )$, hence the result follows.

\item Let $c \in \mathcal{O}_{\gamma, p^i}$, where $ \gamma - \beta +j+1\leq i \leq \gamma - 1.$ Then $c= p^i c'$ for some $c'$, where $(c',p) = 1.$ Clearly $p^{\gamma}\mathbb{Z} \subset [(a,b,c): G].$
Now by considering $y \in [(a,b,c):G]$, we have $y (0 ,0, 1)= n (0,b,c)$ for some integer $n$. This implies $p^{\beta -j}\mid n$ and since $i+\beta - j \geq \gamma +1$, so $y = p^{\beta+i-j}c' (\bmod p^{\gamma} )$.
\end{itemize}

\item Consider the case $a \in \mathcal{O}_{\alpha, p^k}$.  Assume that $a= p^ka'$ for some $a'$, where $(a',p)=1$ for $0\leq k \leq \alpha-1$  and $b\in \mathcal{O}_{\beta, p^{\beta}},$
\begin{itemize}
\item For $c \in \mathcal{O}_{\gamma, p^{\gamma}},$ it is trivially true that $[(a,b,c):G]= p^{\gamma}\mathbb{Z}.$\\
\item  Let $c \in \mathcal{O}_{\gamma, p^i}$, where $0\leq i \leq \beta - \alpha +k $. So $c= p^ic'$ for some $c'$, where $(c',p)=1. $ For any $y = p^{\beta}k' \in p^{\beta}\mathbb{Z},$\\ 
$$y (u,v,w) = (0, 0, p^\beta) k'w = k'wc'^{-1}(p^{\beta-i+k}a', 0 , p^{\beta-i} p^i c') 
\in \Z (a,0,c),$$
since $\beta -i+k \ge \alpha.$ So, $p^{\beta}\mathbb{Z} \subset [(a,b,c): G].$
Now by considering $y \in [(a,b,c):G]$, we have $y (0 ,1, 0)= n (a,0,c)$ for some integer $n$, so we get 
$y\in p^{\beta}\Z$ and we are done.
\item Let $c \in \mathcal{O}_{\gamma, p^i}$, where 
$\beta - \alpha +k +1\leq i \leq \gamma - \alpha +k-1.$ 
Then $c= p^i c'$ for some $c'$, where $(c',p) = 1.$ 
For any $y = p^{i+\alpha -k}k' \in p^{i+\alpha-k}\Z,$\\
$$y (u,v,w) = (0, 0, p^{\alpha+i-k}c') c'^{-1}k'w
=(p^{\alpha-k}p^ka',0,p^{\alpha-k} p^i c') c'^{-1}k'w
\in \Z(a,b,c),$$ since $i+\alpha-k \geq \beta+1.$
So, $p^{\alpha+i -k}\mathbb{Z} \subset [(a,b,c): G].$
Now by considering $y \in [(a,b,c):G]$, we have $y (0 ,0, 1)= n (p^ka',0,p^ic')$ for some integer $n$. This implies $p^{\alpha -k}\mid n$ and $y = p^{\alpha - k +i}n' (\bmod p^{\gamma} )$ for some integer $n'$. Since $\alpha - k +i \leq \gamma -1,$ so $y\in p^{\alpha - k +i} \Z.$ 

\item Let $c \in \mathcal{O}_{\gamma, p^i}$, where $ \gamma -  \alpha+k\leq i \leq \gamma - 1.$ 
Then $c= p^i c'$ for some $c'$, where $(c',p) = 1.$  Clearly $p^{\gamma}\mathbb{Z} \subset [(a,b,c): G].$

Now by considering $y \in [(a,b,c):G]$, we have $y (0 ,0, 1)= n (a,0,c)$ for some integer $n$. This implies $p^{\alpha - k}\mid n$ and since $\alpha -k+i \geq \gamma$, so $y = p^inc' (\bmod p^{\gamma} )$.

\end{itemize}
\end{itemize}

\begin{itemize}

\item Consider the case $a \in \mathcal{O}_{\alpha, p^k}$. Assume that $a= p^ka'$ for some $a'$, where $(a',p)=1$ for $0\leq k \leq \alpha-1$  and $b\in \mathcal{O}_{\beta, p^{j}},$ so $b=p^jb'$
for some $b'$, where $(b',p)=1$ for $0 \leq j \leq \beta-1.$

\begin{itemize}
\item For $c \in \mathcal{O}_{\gamma, p^{\gamma}},$ it is trivially true that $[(a,b,c):G]= p^{\gamma}\mathbb{Z}.$\\

\item  Let $c \in \mathcal{O}_{\gamma, p^i}$, where $0\leq i \leq \gamma - \beta +j $. So $c= p^ic'$ for some $c'$, where $(c',p)=1. $ Note that $\beta-i \geq \alpha.$
Then for any $y = p^{\beta}k' \in p^{\beta}\mathbb{Z},$\\ 
$$y (u,v,w) = (0, 0, p^\beta) k'w =
k'c'^{-1}w (p^{\beta-i}p^k a', p^{\beta-i}p^j b', p^{\beta-i} p^i c') \in \Z(a,b,c).
$$ So $p^{\beta}\mathbb{Z} \subset [(a,b,c): G].$

Now by considering $y \in [(a,b,c):G]$, we have $y (0 ,1, 0)= n (a,b,c)$ for some integer $n.$ Then $p^{\gamma-i} \mid n$ and $y\equiv p^jc' (\bmod)$, so $y \in p^{\beta} \Z$, since $\gamma + (j-i) > \beta.$

\item  Let $c \in \mathcal{O}_{\gamma, p^i}$, where $0\leq i \leq \gamma- \beta +j $. So, $c= p^ic'$ for some $c'$, where $(c',p)=1. $. Since $i> j$, so $\beta-j \geq \alpha.$
Then for any $y = p^{\beta +i -j}k' \in p^{\beta+i-j}\mathbb{Z},$\\ 

$$y (u,v,w) = (0, 0, p^{\beta+i-j}c') k' w=
k' w(0, p^{\beta-j}p^jb', p^{\beta+i-j}c') $$
$$= k'w (p^{\beta-j}p^ka', p^{\beta-j}p^jb', p^{\beta+i-j}c') 
 \in \Z(a,b,c)
$$
as $ i>j.$ So $p^{\beta}\mathbb{Z} \subset [(a,b,c): G].$

Now by considering $y \in [(a,b,c):G]$, we have $y (0 ,0, 1)= n (a,b,c)$ for some integer $n.$ Then $p^{\beta-j} \mid n$ and $y\equiv p^ic' (\bmod)$, so $y \in p^{\beta+i-j} \Z$, since $i< \gamma -\beta +j.$

\item Let $c \in \mathcal{O}_{\gamma, p^i}$, where 
$i \geq j.$ Then $c= p^i c'$ for some $c'$, where $(c',p) = 1.$ Also assume $i>\beta -\alpha >j.$
For any $y = p^{i+\beta -j}k' \in p^{i+\beta-j}\Z,$\\
$$y (u,v,w) = (0,0,p^{\beta+i-j} c')c'^{-1}k'wa'
=(p^{\beta-j}p^ka',p^{\beta-j}p^jb',p^{\beta+i-j} c')k'wc'^{-1}
\in \Z(a,b,c)$$ as $\beta-\alpha >j.$
So, $p^{\beta+i -j}\mathbb{Z} \subset [(a,b,c): G].$

Now by considering $y \in [(a,b,c):G]$, we have $y (0 ,0, 1)= n (p^ka',p^jb',p^ic')$ for some integer $n$. This implies $p^{\beta-j}\mid n$ and $y = p^{i}nc' (\bmod p^{\gamma}) $. Since $i < \gamma -\beta+j,$ so $y\in p^{\beta+i - j +} \Z.$.

\item Let $c \in \mathcal{O}_{\gamma, p^i}$, where  $i \geq j.$ Then $c= p^i c'$ for some $c'$, where $(c',p) = 1.$ Also assume $\beta -\alpha < i$ , $i <\gamma -\alpha+k$ and $j> \beta -\alpha+k.$
For any $y = p^{i+\alpha -k}k' \in p^{i+\alpha-k}\Z,$\\
$$y (u,v,w) = (0,0,p^{\alpha+i-k} c')k'wc'^{-1}
=(p^{\alpha-k}p^k a',p^{\alpha-k}p^jb',p^{\alpha+i-k} c')k'wc'^{-1}
\in \Z(a,b,c)$$ 
as $\alpha -k  > \beta -j.$
So, $p^{\alpha+i -k}\mathbb{Z} \subset [(a,b,c): G].$

Now by considering $y \in [(a,b,c):G]$, we have $y (0 ,0, 1)= n (p^ka',p^jb',p^ic')$ for some integer $n$. This implies $p^{\alpha-k}\mid n$ and $y = p^{i}nc' (\bmod p^{\gamma}) $. 
Since $i < \gamma - \alpha+k,$ so $y\in p^{\alpha+i - k} \Z.$
\item Let $c \in \mathcal{O}_{\gamma, p^i}$, where $i \geq j.$ Then $c= p^i c'$ for some $c'$, where $(c',p) = 1.$ Also assume $\beta -\alpha < i$ , $\gamma -\alpha+k \leq i$ and $j> \beta -\alpha+k.$
From the previous argument $i \geq \gamma -\alpha +k$, therefore we have $p^{\gamma}\Z = [(a,b,c):G].$

\item Let $c \in \mathcal{O}_{\gamma, p^i}$, where $i \geq j.$ Then $c= p^i c'$ for some $c'$, where $(c',p) = 1.$ Also assume  $\beta - \alpha < j \leq  \beta -\alpha+k.$ For any $y = p^{i+\beta -j}k' \in p^{i+\beta - j}\Z,$ \\
$$y (u,v,w) = 
(0,0,p^{\beta+i-j} c')k'wc'^{-1}
=(p^{\beta-j+k}a',p^{\beta-j}p^jb',p^{\beta+i-j} c')k'wc'^{-1}
\in \Z(a,b,c)$$ 
as $\beta -\alpha < j \leq \beta -\alpha +k.$
So, $p^{\beta+i -j}\mathbb{Z} \subset [(a,b,c): G].$

Now by considering $y \in [(a,b,c):G]$, we have $y (0 ,0, 1)= n (p^ka',p^jb',p^ic')$ for some integer $n$. This implies $p^{\beta-j}\mid n$ and $y = p^{i}nc' (\bmod p^{\gamma}) $. 
Since $i < \gamma - \beta+j,$ so $y\in p^{\beta+i - j} \Z.$

\item Let $c \in \mathcal{O}_{\gamma, p^i}$, where $i < j.$ Then $c= p^i c'$ for some $c'$, where $(c',p) = 1.$ Also assume  $ i>  \beta -\alpha$ and $\beta +k-j \geq \alpha.$

For any $y = p^{\beta }k' \in p^{\beta }\Z,$\\
$$y (u,v,w) = 
(0,0,p^{\beta-i}p^ic')k'wc'^{-1}
=(p^{\beta-i}p^ka',p^{\beta-i}p^jb',p^{\beta-i}p^i c')k'wc'^{-1}
\in \Z(a,b,c)$$ 
as $\beta -\alpha < i.$
So, $p^{\beta}\mathbb{Z} \subset [(a,b,c): G].$

Now by considering $y \in [(a,b,c):G]$, we have $y (0 ,1, 0)= n (p^ka',p^jb',p^ic')$ for some integer $n,$. This implies $p^{\gamma-i}\mid n$ and $y = p^{j}nb' (\bmod p^{\gamma} )$. 
Since $i < \gamma - \beta+j,$ so $y\in p^{\beta} \Z.$

\item Let $c \in \mathcal{O}_{\gamma, p^i}$, where $i < j.$ Then $c= p^i c'$ for some $c'$ such that $(c',p) = 1.$ Also assume  $ \beta -\alpha+k>i>  \beta -\alpha$ and 
$j> \beta -\alpha+k.$ For any $y = p^{\beta }k' \in p^{\beta }\Z,$ then
$$y (u,v,w) = 
(0,0,p^{\beta-i}p^ic')k'wc'^{-1}
=(p^{\beta-i}p^ka',p^{\beta-i}p^jb',p^{\beta-i}p^i c')k'wc'^{-1}
\in \Z(a,b,c)$$ 
as $\beta -\alpha +k> i.$
So $p^{\beta}\mathbb{Z} \subset [(a,b,c): G].$

Now by considering $y \in [(a,b,c):G]$, we have $y (0 ,1, 0)= n (p^ka',p^jb',p^ic')$ for some integer $n$. This implies $p^{\alpha-k}\mid n$ and $y = p^{j}nc' (\bmod p^{\beta}) $. 
Since $j >  \beta -\alpha+k,$ so $y\in p^{\beta} \Z.$

\item Let $c \in \mathcal{O}_{\gamma, p^i}$, where 
$i < j.$  Then $c= p^i c'$ for some $c'$, where $(c',p) = 1.$ Also assume  $\gamma -\alpha+k> i \geq  \beta -\alpha +k$. For any $y = p^{\alpha+i-k }k' \in p^{\alpha+i-k }\Z,$ \\
$$y (u,v,w) = 
(0,0,p^{\alpha-k}p^ic')k'wc'^{-1}
=(p^{\alpha-k}p^ka',p^{\alpha-k}p^jb',p^{\alpha-k}p^i c')k'wc'^{-1}
\in \Z(a,b,c)$$ 
as $\alpha -k+j> \beta.$
So, $p^{\alpha+i-k}\mathbb{Z} \subset [(a,b,c): G].$

Now by considering $y \in [(a,b,c):G]$, we have $y (0 ,0, 1)= n (p^ka',p^jb',p^ic')$ for some integer $n$. This implies $p^{\alpha-k}\mid n$ and $y = p^{i}nc' (\bmod p^{\gamma})$. 
Since $i <  \gamma -\alpha+k,$ so $y\in p^{\alpha-k+i} \Z.$

\item Let $c \in \mathcal{O}_{\gamma, p^i}$, where 
$i < j.$ 
Then $c= p^i c'$ for some $c'$, where $(c',p) = 1.$ Also assume  $i \geq \gamma -\alpha+k $
and $j>  \beta -\alpha +k$. By assuming that $i \geq \gamma -\alpha+k $ and follow the same proof as above, we trivially have the required result.

\item Let $c \in \mathcal{O}_{\gamma, p^i}$ where $ \gamma -  \beta+j\leq i \leq \gamma - 1.$ 
Then $c= p^i c'$ for some $c'$, where $(c',p) = 1.$ Clearly $p^{\gamma}\mathbb{Z} \subset [(a,b,c): G].$
Now by considering $y \in [(a,b,c):G]$, we have $y (0 ,0, 1)= n (a,b,c)$ for some integer $n$. This implies 
$p^{\beta - j}\mid n$ and since $\beta +i -j \geq \gamma$, so $y = p^inc' (\bmod p^{\gamma} )$.

This completes the proof.

\end{itemize}

\end{itemize}

\end{proof}
 
\end{document}